\documentclass[11pt]{amsart}

\usepackage{hyperref}
\usepackage{amsfonts}
\usepackage{amssymb}
\usepackage{amsmath,mathtools}
\usepackage{latexsym}
\usepackage{amscd}
\usepackage{mathrsfs}
\usepackage{enumitem} 
\usepackage{braket}
\usepackage[margin=1.5in]{geometry}
\usepackage{amsthm}

\newtheorem{prop}{Proposition}[section]
\newtheorem{theo}[prop]{Theorem}
\newtheorem{lemm}[prop]{Lemma}
\newtheorem{coro}[prop]{Corollary}

\theoremstyle{definition}
\newtheorem{defi}[prop]{Definition}

\newcommand{\RR}{\mathbb{R}}

\newcommand{\sL}{\mathscr{L}}

\DeclareMathOperator{\tr}{tr}
\newcommand{\bangle}[1]{\left\langle #1 \right\rangle}
\DeclareMathOperator{\Ric}{Ric}
\DeclareMathOperator{\Div}{div}

\allowdisplaybreaks

\title{Expanding Ricci Solitons Asymptotic to Cones}
\author{Otis Chodosh}
\date{\today}
\thanks{ I would like to thank my advisor, Simon Brendle, for suggesting this problem as well as for his invaluable guidance. I am also grateful to Richard Bamler and Leon Simon for discussions about the Ricci flow and asymptotic limit theorems and for their continued encouragement. This work was supported in part by a National Science Foundation Graduate Research Fellowship DGE-1147470. }
\address{Department of Mathematics, Stanford University, Stanford, CA 94305}
\email{ochodosh@math.stanford.edu}

\begin{document}

\begin{abstract} 
We show that an expanding gradient Ricci soliton which is asymptotic to a cone at infinity in a certain sense must be rotationally symmetric. 
\end{abstract}

\maketitle

\section{Introduction}
An expanding gradient soliton is a self-similar solution to the Ricci flow that flows by diffeomorphism and expanding homothety. Recall that expanding gradient solitons satisfy $2 \Ric_{g} + g = \sL_{\nabla f} (g)$ for some function $f$. To relate these two notions, one may let $X : = \nabla f$ and denote by $\Phi_{\tau}$ the gradient flow of $-X$ for time $\tau$. It is not hard to see that the metric $\hat g(t) = t \Phi^{*}_{\log t} (g)$ is a (self-similar) solution to the Ricci flow. An overview of Ricci solitons, along with many existence and uniqueness results may be found in \cite{Cao:RecentSoliton}. 

Expanding Ricci solitons are of interest for a variety of reasons. In addition to providing a natural generalization of Einstein metrics, they model Type-III singularities in the Ricci flow \cite{Hamitlon:Singularities95} (see also \cite{Cao:KRFsolitons,ChenZhu:pinchedCurvature}) and provide examples of equality in Hamilton's Harnack inequality \cite{Hamilton:Harnack}. Finally, we remark that in \cite{SchulzeSimon:Expand-Cones}, Schulze--Simon have recently constructed solutions to Ricci flow coming out of the asymptotic cone at infinity of manifolds with positive curvature operator and shown that such a solution to Ricci flow must be an expanding gradient soliton.  

The simplest example of a non-Einstein expanding gradient soliton is the Gaussian soliton $(\RR^{n}, \delta_{\text{euc}}, |x|^{2}/4)$. In addition, Bryant has constructed non-flat expanding gradient solitons which are rotationally symmetric and are asymptotic to a cone at infinity \cite{Bryant:Solitons}. In particular, he has constructed a one-parameter family of solitons with positive sectional curvature. We recall his construction in the appendix to this work. 

In dimensions $n\geq 4$, there are additional known examples of expanding solitons. For example, Cao has constructed $U(n)$-invariant expanding gradient K\"ahler-Ricci solitons in \cite{Cao:KRFsolitons}. These were generalized to a construction of expanding gradient K\"ahler-Ricci solitons on complex line bundles by Feldman--Ilmanen--Knopf in \cite{FeldmanIlmanenKnopf}. Additionally, Gastel--Kronz have constructed doubly warped (non-K\"ahler) expanding gradient solitons in \cite{GastelKronz}. 

Various authors have obtained uniqueness results concerning expanding gradient solitons. In \cite{ChenZhu:pinchedCurvature}, Chen--Zhu show that a non-compact expanding gradient soliton with positive sectional curvature and uniformly pinched Ricci curvature must be the flat expanding Gaussian soliton. Another uniqueness result for the Gaussian soliton can be found in the work of Pigola--Rimoldi--Setti \cite{PigolaRimoldiSetti:remarksSol} in which they show that under certain integrability assumptions of the soliton vector field, the soliton must be flat. Additionally, Cao--Catino--Chen--Mantegazza--Mazzeiri have shown that an expanding gradient soliton with positive Ricci curvature must be rotationally symmetric under certain assumptions on the Bach tensor (that it is divergence free when $n=3$ and that it vanishes identically when $n\geq 4$) \cite{CaoCatinoChenMantegazzaMazzieri:BachFlat}. We also remark that Chen and Chen--Deruelle have studied the asymptotic geometry of expanding solitons with finite asymptotic curvature ratio, showing that they have cone structure at infinity \cite{Chen:AsymptoticExpanding,ChenDerulle:StructInftyExpandSol}.

In order to state our results, we let $g_{\alpha}$ denote the conical metric with cone angle $\alpha \in [0,1)$ on $\RR^{n}\backslash\{0\}$ given in polar coordinates by $g_{\alpha}:= dr^{2}+ (1-\alpha)r^{2} g_{S^{n-1}}$. We thus define
\begin{defi}\label{defi:asymp-cone-sol}
We say that an expanding gradient soliton $(M,g,f)$ is \emph{asymptotically conical as a soliton} if there is a map $F:  (r_{0},\infty)_{r} \times (S^{n-1})_{\omega}\to M$ so that
\begin{enumerate}
	\item $F$ is a diffeomorphism onto its image and $M\backslash F^{-1}((r_{0},\infty) \times S^{n-1})$ is a compact set. 
	\item It parametrizes the level sets of $f$, in the sense that $f(F(r,\omega)) = r^{2}/4$ and 
	\begin{equation*}
		\frac{\partial F}{\partial r} = \sqrt{f} \frac{X}{|X|^{2}}.
	\end{equation*}
	\item In these coordinates, $g$ is $C^{2}$-asymptotic to a conical metric, in the sense that $F^{*}(g) = g_{\alpha} + k$ for some $\alpha \in [0,1)$ and $k$ some $(0,2)$-tensor so that $|\nabla^{j}k|= O(r^{-3\epsilon-j})$ for some $\epsilon>0$ and $j=0,1,2$.
\end{enumerate}
\end{defi}

In particular, Bryant's expanding 1-parameter family of positively curved solitons satisfy these assumptions by Proposition \ref{prop:app-bry-are-asymp-cone}. The goal of this paper is to prove:
\begin{theo}\label{theo:main-unique-theo}
Suppose that $(M^{n},g,f)$ is an expanding gradient soliton (for $n\geq 3$) that has positive sectional curvature and is asymptotically conical as a soliton, as defined above. Then, $(M,g,f)$ is rotationally symmetric. 
\end{theo}

We point out that for any soliton $(M,g,f)$ with non-negative Ricci curvature, $f$ behaves asymptotically like $d(p,\cdot)^{2}/4$ (where $p$ is the point where $f$ attains its minimum value) so the second assumption is quite general. On the other hand, the first assumption, which determines the topology of $M$ near infinity, is more restrictive yet still natural. In particular, we note that if $(M,g)$ has positive curvature operator, then its asymptotic cone at infinity (in the Gromov--Hausdorff sense) is a cone over a manifold homeomorphic to $S^{n-1}$. This follows from works by Kapovitch and Perelman, as explained in \cite[Appendix 2]{SchulzeSimon:Expand-Cones}.

Our proof of Theorem \ref{theo:main-unique-theo} is based on the recent works of Brendle \cite{Brendle:3DSolitonUniqueness,Brendle:HighDimSoliton} in which it is shown that a steady Ricci soliton with positive sectional curvature that parabolically blows down to a shrinking cylinder must be rotationally symmetric. In particular, assuming the soliton is $\kappa$-noncollapsed, these assumptions are always satisfied in three dimensions, answering a question raised in Perelman's first paper \cite{Perelman:Entropy}:

\begin{theo}[S. Brendle {\cite{Brendle:3DSolitonUniqueness}}]
In three dimensions, a $\kappa$-noncollapsed complete non-flat steady gradient soliton must be the rotationally symmetric Bryant soliton. 
\end{theo}

 On the other hand, there are several crucial differences between the arguments used in \cite{Brendle:3DSolitonUniqueness,Brendle:HighDimSoliton} to handle the steady case and those of the current paper. In particular, as it does not seem possible to perform a parabolic blowdown of the expanding solitons under consideration, in most parts of the paper all that we have at our disposal is the elliptic maximum principle. However, we are fortunate to have more effective barriers in this case, and these turn out to be sufficient to replace the blowdown arguments used by Brendle to handle the steady case.

We also remark that Kotschwar--Wang have recently demonstrated a rigidity property of shrinking Ricci solitons asymptotic to cones \cite{KotschwarWang:shrinkers}. They make use of methods that are quite different from those used in this work, relying on techniques used to prove backwards uniqueness for heat-type equations. Interestingly, they do not require any curvature assumptions (which would be somewhat unnatural in the setting of asymptotically conical shrinkers: a shrinker with nonnegative Ricci curvature cannot be asymptotically conical unless it is flat) and apply to arbitrary cones. 

However, it seems that the asymptotically conical and shrinker assumptions are crucial for their techniques to apply, and as such it seems unlikely that they would be applicable to the case of expanding solitons considered in this work. For example, there are incomplete expanding solitons that are asymptotically conical as solitons,\footnote{This follows from a straightforward modification of the proof in the appendix that the expanding Bryant soliton exists and has the desired properties.} in the sense of our Definition \ref{defi:asymp-cone-sol}. In contrast, it is not possible that an incomplete shrinker is asymptotic to the same cone as a complete one, by Kotschwar--Wang's result.  

We now describe the structure of the paper. In Section \ref{sect:asympt-geo}, we collect several results about the behavior of the soliton in the asymptotically conical region. Then, in Section \ref{sect:derEqns}, we show that if a vector field satisfies $\Delta W + D_{X} W - \frac 12 W = 0$, then $h := \sL_{W}(g)$ satisfies $\Delta_{L} h + \sL_{X}h - h =0$. A crucial observation is that both of these equations have lowest order term of the correct sign in order to apply the maximum principle. In particular, we will later show that certain solutions $W$ to the first PDE are Killing vector fields on the expanding soliton, by showing that $\sL_{W}(g)$ vanishes identically, thanks to the maximum principle applied to the second PDE. 

In Section \ref{sect:maxAKVF}, we observe that the function $\left(f +\frac n 2\right)^{-\epsilon}$ acts as a barrier for the PDE on vector fields described above. We then use this to construct a vector field $V$ solving $\Delta V + D_{X} V - \frac 12 V = Q$ which has $|V|,|DV| = O(r^{-2\epsilon})$. Here, $Q$ is a given vector field with $|Q| = O(r^{-2\epsilon})$. Then, in Section \ref{sect:lichPDE}, a barrier argument using $2\Ric + g$ is used to show that any solution to $\Delta_{L} h + \sL_{X}(h) - h =0$ with $|h| = o(1)$ must vanish identically. 

Finally, the proof of Theorem \ref{theo:main-unique-theo} is given in Section \ref{sect:proofMainTheo}. The main idea of the proof is to consider approximate Killing vector fields at infinity coming from symmetries of the exact cone and perturb them so as to be actual Killing vector fields. More precisely, if a vector field $U$ satisfies $|\Delta U + D_{X}U -\frac 12 U| \leq O(r^{-2\epsilon})$ as well as $|\sL_{U}(g)|\leq O(r^{-2\epsilon})$, then using the results in Section \ref{sect:maxAKVF}, we may find a vector field $V$ so that $W:= U-V$ satisfies $\Delta W + D_{X} W - \frac 12 W = 0$ and $h:=\sL_{W}(g)$ decays as $|h|\leq O(r^{-2\epsilon})$. By the results in Section \ref{sect:derEqns}, we then see that $h$ satisfies $\Delta_{L} h + \sL_{X}(h) -\frac 12 h = 0$. Finally, the results in Section \ref{sect:lichPDE} show that $h$ must vanish identically, so $W$ is a Killing vector field. It is not hard to show that this allows us to upgrade the approximate Killing vectors to exact Killing vectors, showing rotational symmetry. 

We also include in the appendix a discussion of the rotationally symmetric expanding gradient solitons constructed by Bryant in \cite{Bryant:Solitons}. In particular, we prove that the expanding Bryant solitons are asymptotically conical as solitons, in the sense of Definition \ref{defi:asymp-cone-sol}, by showing that the non-convergent formal power series for the soliton warping function obtained in \cite{Bryant:Solitons} gives the correct asymptotics for the warping function, along with its first two derivatives. 
 
\section{Asymptotic geometry} \label{sect:asympt-geo}

By an observation of Hamilton \cite{Hamitlon:Singularities95}, $|\nabla f|^{2} + R- f $ is constant. We will assume throughout the paper that 
\begin{equation}
|\nabla f|^{2}+R= f. 
\end{equation}
Combined with the trace of the soliton equation this gives 
\begin{equation}
\Delta f + |\nabla f|^{2} = \frac {n}{2} + f.
\end{equation}
Observe that $\Ric_{g_{\alpha}} =  (n-2)\alpha g_{S^{n-1}}$. From the formula for $D\Ric|_{g_{\alpha}}(k)$ (cf.\ \cite[Theorem 1.174(d)]{besse}) we see that $\Ric_{g} =(n-2)\alpha g_{S^{n-1}} + O(r^{-2\epsilon})$. Hamilton's identity thus yields 
\begin{equation}
|\nabla f|^{2} = f+ O(r^{-2\epsilon}).
\end{equation}

We also estimate $\Phi_{\tau}$ in the asymptotic region 
\begin{equation*}\begin{split}
f(\Phi_{\tau}(p)) - f(0) & = \int_{0}^{\tau} df|_{\Phi_{s}(p)} (-X) ds\\
& =- \int_{0}^{\tau} |\nabla f|^{2} (\Phi_{s}(p))  ds\\
& = - \int_{0}^{\tau}( f(\Phi_{s}(p)) - R(\Phi_{s}(p)))ds\\
& \geq - \int_{0}^{\tau} f(\Phi_{s}(p)) ds.
\end{split}
\end{equation*}
The integrated Gr\"onwall's identity thus implies that $f(\Phi_{\tau}(p)) \geq f(p) e^{-\tau}$. As such, in the asymptotic region, we have that (writing $r(\cdot)$ for the radial coordinate in the asymptotic region)
\begin{equation}\label{eq:est-r-phi}
r(\Phi_{\tau}(p)) \geq  r(p) e^{-\tau/2}.
\end{equation}

Finally, we will need an estimate for $|\nabla R|$. By the conical asymptotics, $R = O(r^{-3\epsilon})$. Furthermore, because the asymptotics imply that $g$ has bounded curvature, we may apply Shi's local estimates in balls of unit radius and on the time interval $[1/2,1]$ to see that $|D^{m}R| \leq C_{m}$. Thus, using Hamilton's tensor interpolation inequalities \cite{Hamilton:3mflds} (and the fact that the Sobolev constant for balls of unit radius is bounded as $r\to\infty$) we see that 
\begin{equation}\label{eq:bd-nabla-R}
|\nabla R| = O(r^{-2\epsilon}).
\end{equation}

\section{A Lichnerowicz PDE for the Lie derivative of approximate KVFs}\label{sect:derEqns}

It will be important to recall that for any $(0,2)$-tensor $h$, and frame $\{e_{1},\dots,e_{n}\} \in T_{p}M$, the Lichnerowicz Laplacian is defined as
\begin{equation*}\begin{split}
(\Delta_{L} h)(e_{i},e_{k}) & = (\Delta h) (e_{i},e_{k}) + 2\sum_{j,l=1}^{n}R(e_{i},e_{j},e_{k},e_{l}) h(e_{j},e_{l})\\
& \qquad - h(\Ric(e_{i}),e_{k}) - h(e_{i},\Ric(e_{k})),
\end{split}\end{equation*}
where $\Delta = - \nabla^{*}\nabla$ is the usual ``rough'' connection Laplacian. By \cite[Proposition 2.3.7]{topping:RF}, for any vector field $W$
\begin{equation*}
\sL_{W}(\Ric)= -\frac 12 \Delta_{L} h + \frac 12 \sL_{Z}(g)
\end{equation*}
where $h = \sL_{W}(g)$ and 
\begin{equation*}
Z = \Div h - \frac 12 \nabla (\tr h) = \Delta W + \Ric(W).
\end{equation*}
By the soliton equation and the fact that $[\sL_{W},\sL_{X}]=\sL_{[W,X]}$
\begin{equation*}
\Delta_{L} h+ \sL_{X}(h)- h =  \sL_{Z-[W,X]}(g).
\end{equation*}
We compute
\begin{equation*}
\begin{split}
Z - [W,X] & = \Delta W + \frac 12 \sL_{X}(g)(W) - \frac 12 W - [W,X]\\
& = \Delta W + D_{W} X - [W,X] - \frac 12 W\\
& = \Delta W + D_{X} W - \frac 12 W.
\end{split}
\end{equation*}
Thus we have shown:
\begin{prop}\label{prop:lich-lie-deriv-AKVF}
If $W$ satisfies $\Delta W + D_{X}W - \frac 12 W = 0$, then $h = \sL_{W}(g)$ satisfies
\begin{equation*}
\Delta_{L} h + \sL_{X} h -h = 0.
\end{equation*}
\end{prop}

We will later use the following corollary to show that $2\Ric + g$ is a barrier for $\Delta_{L} h + \sL_{X}h - h = 0$.
\begin{coro}\label{coro:lich-eqn-Ric}
The vector field $X$ satisfies $\Delta X + D_{X}X - \frac 12 X = 0$ and thus $\sL_{X}(g) = 2\Ric + g$ satisfies
\begin{equation*}
\Delta_{L} (2\Ric + g) + \sL_{X}(2\Ric +g) -( 2\Ric + g)=0.
\end{equation*}
\end{coro}

\begin{proof}
The vanishing of 
\begin{equation*}
Z - [X,X]  = \Div(\sL_{X}(g)) - \frac 12 \nabla (\tr \sL_{X}(g)) = 2 \Div(\Ric) - \nabla R
\end{equation*}
follows from the contracted second Bianchi identity. 
\end{proof}

Alternatively, one could check that the conclusion of the above corollary is equivalent to the simpler equation $\Delta_{L}(\Ric) + \sL_{X}(\Ric) = 0$, which follows from the evolution of the Ricci tensor under Ricci flow as well as scale invariance of the Ricci tensor. We have left the equation in its more complicated form because this is how it will be used in the sequel.

\section{A Maximum principle for approximate KVFs} \label{sect:maxAKVF}

Suppose that $Q$ is a vector field on $M$ such that $|Q| = O(r^{-2\epsilon})$ for some $\epsilon < \frac {1}{\sqrt{2}}$. In this section, we will solve for a smooth vector field $V$ satisfying $\Delta V+D_{X}V - \frac 12 V = Q$ with $|V| = O(r^{-2\epsilon})$ and $|DV|_{g} = O(r^{-2\epsilon})$. 
 
 \begin{lemm}\label{lemm:max-princ-AKVF}
If $V$ satisfies $\Delta V+ D_{X}V - \frac 12 V = Q$ in $\{f \leq \rho^{2}\}$, then 
\begin{equation*}
\sup_{\{f\leq \rho^{2}\}}\left( |V| - B \left(f+\frac n 2\right)^{-\epsilon} \right) \leq \max \left\{\sup_{\{f=\rho^{2}\}} |V| - B\left(\rho^{2}+\frac n 2\right)^{-\epsilon}, 0\right\}
\end{equation*}
for some uniform constant $B> 0$. 
\end{lemm}

\begin{proof}
By the identities discussed in Section \ref{sect:asympt-geo}, $\Delta f + |\nabla f|^{2} = f + \frac n 2$ and $|\nabla f|^{2} + R = f$, (the latter implying that $f + \frac n 2 > 1$) we have that
\begin{align*}
& \Delta \left(f+\frac n 2\right)^{-\epsilon} + D_{X} \left(f+\frac n 2\right)^{-\epsilon} - \frac 12 \left(f+\frac n 2\right)^{-\epsilon}  \\
& = - \epsilon \left( f+\frac n 2\right)^{-\epsilon - 1} \Delta f +  \epsilon(\epsilon + 1) \left(f+\frac n 2\right)^{-\epsilon-2} |\nabla f|^{2}\\
& \qquad - \epsilon \left(f+\frac n 2\right)^{-\epsilon-1} |\nabla f|^{2} - \frac 12 \left(f+\frac n 2\right)^{-\epsilon}\\
& = - \epsilon \left( f+\frac n 2\right)^{-\epsilon - 1} \left( f + \frac n 2 \right) +  \epsilon(\epsilon + 1) \left(f+\frac n 2\right)^{-\epsilon-2} |\nabla f|^{2} - \frac 12 \left(f+\frac n 2\right)^{-\epsilon}\\
& = -\left( \epsilon+ \frac 12 \right) \left( f+\frac n 2\right)^{-\epsilon } +  \epsilon(\epsilon + 1) \left(f+\frac n 2\right)^{-\epsilon-2} |\nabla f|^{2} \\
& = -\left( \epsilon+ \frac 12 \right) \left( f+\frac n 2\right)^{-\epsilon } +  \epsilon(\epsilon + 1) \left(f+\frac n 2\right)^{-\epsilon-1} \\
& \qquad - \epsilon(\epsilon + 1) \left(f+\frac n 2\right)^{-\epsilon-2}\left(\frac n 2 +R\right) \\
& < -\left(  \frac 12 -  \epsilon^{2} \right) \left( f+\frac n 2\right)^{-\epsilon } 
\end{align*}
Because $|Q| = O(r^{-2\epsilon})$, we see that we may find $B >0$ so that 
\begin{equation*}
|Q| \leq B \left(  \frac 12 -  \epsilon^{2} \right) \left( f+\frac n 2\right)^{-\epsilon }.
\end{equation*}
We define the quantity $\varphi: = |V| - B\left(f+\frac n 2\right)^{-\epsilon}$. It is easy to check (cf.\ \cite[Proposition 5.1]{Brendle:3DSolitonUniqueness}) that Kato's inequality implies that 
\begin{equation*}
\Delta |V| + D_{X}|V| -\frac 12 |V| \geq - |Q|
\end{equation*}
when $V \not = 0$. By our choice of $B$, this implies
\begin{equation*}
\Delta \varphi + D_{X} \varphi - \frac 12 \varphi \geq 0
\end{equation*}
at all points where $\varphi \geq 0$. We may thus apply the maximum principle to $\varphi$.
 \end{proof}

\begin{prop}\label{prop:solve-for-vf-decay}
Still assuming that $|Q| = O(r^{-2\epsilon})$ (for $\epsilon < \frac {1}{\sqrt{2}}$), we may find a vector field $V$ which solves 
\begin{equation*}
\Delta V + D_{X} V - \frac 12 V = Q
\end{equation*}
on all of $M$, and so that $|V| = O(r^{-2\epsilon}),|DV| = O(r^{-2\epsilon})$.

\end{prop}
 \begin{proof}

For a given $\rho_{m}\to \infty$, we may solve the Dirichlet problem 
\begin{equation*}
\begin{dcases}
\Delta V^{(m)} + D_{X} V^{(m)} - \frac 12 V^{(m)} = Q & \text{in $\{f\leq \rho_{m}^{2}\}$}\\
V^{(m)} = 0 & \text{on $\{f=\rho_{m}^{2}\}$}.
\end{dcases}
\end{equation*}
Applying Lemma \ref{lemm:max-princ-AKVF}, 
\begin{equation*}
|V^{(m)}| \leq B\left(f ^{2}+ \frac n 2\right)^{-\epsilon} \leq B \left( \frac n 2\right)^{-\epsilon}
\end{equation*}
on $\{f\leq \rho_{m}^{2}\}$. Furthermore, elliptic estimates show that the $|DV^{(m)}|$ are uniformly bounded on compact sets (along with higher derivatives). Thus, extracting a subsequence, we may take $m\to\infty$ to find a smooth vector field $V$ solving 
\begin{equation*}
\Delta V + D_{X} V- \frac 12 V = Q
\end{equation*}
with $|V|=O(r^{-2\epsilon})$. It thus remains to bound $|DV|$, which we now do by parabolic interior estimates. 
 
Recall that $\hat g(t) = t \Phi^{*}_{\log(t)} (g)$ is a solution to the Ricci flow. We define 
 \begin{equation*}
 \hat V(t) : = \Phi^{*}_{\log(t)}(V)
 \end{equation*}
 and
 \begin{equation*}
 \hat Q(t) := t^{-1} \Phi^{*}_{\log(t)} (Q).
 \end{equation*}
It is easy to see that $\hat V$ satisfies the parabolic PDE
\begin{equation}\label{eq:par-pde-V}
\frac{\partial}{\partial t} \hat V = \Delta_{\hat g(t)} \hat V +\Ric_{\hat g(t)} (\hat V)- \hat Q.
\end{equation}
 Fixing a sequence $r_{m}\to\infty$, we may use standard interior parabolic gradient estimates to conclude that
 \begin{equation*}
\begin{split}
\sup_{\{r=r_{m}\}} |DV| & = \sup_{\{r=r_{m}\}} |D\hat V|_{\hat g(1)} \\
& \leq C \sup_{t\in[1/2,1]}\sup_{\{r_{m} -1\leq r \leq r_{m}+1\}} |\hat V |_{\hat g(t)}\\
& + C  \sup_{t\in[1/2,1]}\sup_{\{r_{m} -1\leq r \leq r_{m}+1\}} |\hat Q|_{\hat g(t)}.
 \end{split}
  \end{equation*}
  We remark that the parabolic estimates apply with a uniform constant because we may control the ellipticity and lower order terms in \eqref{eq:par-pde-V} using the asymptotics of $g$. 
  
By the estimate on $r(\Phi_{\tau}(p))$ in the asymptotic region obtained in \eqref{eq:est-r-phi},
\begin{equation*}
\begin{split}
  \sup_{t\in[1/2,1]}\sup_{\{r_{m} -1\leq r \leq r_{m}+1\}} |\hat V|_{\hat g(t)} & =   \sup_{t\in[1/2,1]}\sup_{\{r_{m} -1\leq r \leq r_{m}+1\}} \Phi_{\log(t)}^{*}(|V|) \\
  & \leq \sup_{t\in[1/2,1]} C t^{-2\epsilon}(r_{m}-1)^{-2\epsilon}\\
  & = O(r_{m}^{-2\epsilon}).
  \end{split}
\end{equation*}
An identical argument for the $|\hat Q|_{\hat g(t)}$ term proves that $|DV| = O(r^{-2\epsilon})$.  
 \end{proof}

\section{A Maximum principle for the Lichnerowicz PDE}\label{sect:lichPDE}

The goal of this section is to prove the following proposition, which we will later use to conclude that certain vector fields are actually Killing vector fields.

\begin{prop}\label{prop:lich-eqn-barrier}
Suppose that a $(0,2)$-tensor $h$ satisfies $\Delta_{L} h + \sL_{X}(h) - h = 0$ with $|h| = o(1)$. Then $h\equiv 0$. 
\end{prop}

\begin{proof}
Because $(M,g)$ has positive sectional curvature, $2\Ric + g \geq g$. Thus, by the decay assumption on $h$, we may take $\theta$ large enough so that $\theta(2\Ric + g) \geq h$.
Taking the smallest such $\theta\geq 0$, let 
\begin{equation*}
w := 2\theta \Ric + \theta g - h \geq 0.
\end{equation*}
If $\theta \not = 0$ then there exists a point $p \in M$ and orthonormal basis $\{e_{1},\dots,e_{n}\} \in T_{p}M$ so that at $p$, $w(e_{1},e_{1})=0$, and $(R(e_{1},e_{k},e_{1},e_{l}))_{k,l\in \{1,\dots,n\}}$ is a diagonal matrix. Extending $\{e_{1},\dots,e_{n}\}$ to a local frame near $p$ that is parallel at $p$, the function $w(e_{1},e_{1})$ has a local minimum at $p$, which implies that $(\Delta w) (e_{1},e_{1}) \geq 0$ and $(D_{X} w) (e_{1},e_{1}) = 0$ at $p$. 

Notice that $\Delta_{L} w + \sL_{X} w - w = 0$ by Proposition \ref{prop:lich-lie-deriv-AKVF} and Corollary \ref{coro:lich-eqn-Ric}. For $i\in\{1,\dots,n\}$ evaluating this in the $(e_{i},e_{i})$ direction gives
\begin{align*}
0 & = (\Delta w)(e_{i},e_{i}) + 2\sum_{k,l=1}^{n} R(e_{i},e_{k},e_{i},e_{l}) w(e_{k},e_{l}) -2 w(\Ric(e_{i}),e_{i})\\
& \qquad + \sL_{X} (w(e_{i},e_{i})) - 2w(\sL_{X}e_{i},e_{i}) - w(e_{i},e_{i})\\
 & = (\Delta w)(e_{i},e_{i}) + 2\sum_{k,l=1}^{n} R(e_{i},e_{k},e_{i},e_{l}) w(e_{k},e_{l})- 2w(\Ric(e_{i}),e_{i})\\
& \qquad + D_{X} (w(e_{i},e_{i})) - 2w(D_{X} e_{i} - D_{e_{i}} X,e_{i}) - w(e_{i},e_{i})\\
 & = (\Delta w)(e_{i},e_{i}) + 2\sum_{k,l=1}^{n} R(e_{i},e_{k},e_{i},e_{l}) w(e_{k},e_{l})- 2w(\Ric(e_{i}),e_{i})\\
& \qquad + (D_{X} w)(e_{i},e_{i}) + 2w(D_{e_{i}} X,e_{i}) - w(e_{i},e_{i})\\
 & = (\Delta w)(e_{i},e_{i}) + 2\sum_{k,l=1}^{n} R(e_{i},e_{k},e_{i},e_{l}) w(e_{k},e_{l}) - 2w(\Ric(e_{i}),e_{i})\\
& \qquad + (D_{X} w)(e_{i},e_{i})) +  w(\sL_{X}g(e_{i}),e_{i}) - w(e_{i},e_{i})\\
& = (\Delta w)(e_{i},e_{i})+ (D_{X} w)(e_{i},e_{i}) + 2\sum_{k,l=1}^{n} R(e_{i},e_{k},e_{i},e_{l}) w(e_{k},e_{l}).
\end{align*}
Taking $i=1$ in the above formula, we thus have that at $p$
\begin{equation*}
0 \geq  2\sum_{k,l=1}^{n} R(e_{1},e_{k},e_{1},e_{l}) w(e_{k},e_{l}).
\end{equation*}
Because $(M,g)$ has positive sectional curvature and $(R(e_{1},e_{k},e_{1},e_{l})))_{k,l\in\{1,\dots,n\}}$ is diagonal at $p$, we thus see that $w(e_{k},e_{k}) = 0$ at $p$, for all $k \in \{1,\dots,n\}$. Thus, $\tr w = 0$ at $p$, so $\tr w$ achieves its minimum at $p$.  

Using the fact that the metric is compatible with the connection, the above identity implies that 
\begin{equation*}
\Delta \tr w + D_{X} \tr w = - 2 \sum_{k,l=1}^{n}\Ric(e_{k},e_{l}) w(e_{k},e_{l}) \leq 0.
\end{equation*}
We may thus apply Hopf's strong minimum principle to show that $\tr w \equiv 0$. Considering the asymptotic behavior of $\tr w$, we easily see that $\theta = 0$. Applying the above argument to $-h$ shows that $h\equiv 0$, as desired. 
\end{proof}

\section{Proof of Theorem \ref{theo:main-unique-theo}} \label{sect:proofMainTheo}

First, we use the conical asymptotics to establish the existence of approximate Killing vector fields:
\begin{prop}
There exist vector fields $U_{a}$ for $a \in \left\{ 1 ,\dots,\frac{n(n-1)}{2}\right\}$ so that $|\sL_{U_{a}}| = O(r^{-2\epsilon})$ and $|\Delta U_{a} + D_{X} U_{a} - \frac 12 U_{a} | = O(r^{-2\epsilon})$. Furthermore, $|U_{a}|= O(r)$ and
\begin{equation*}
\sum_{a=1}^{\frac{n(n-1)}{2}} U_{a}\otimes U_{a} = r^{2} \sum_{i=1}^{n-1} \tilde e_{i}\otimes \tilde e_{i} + O( r^{2-2\epsilon})
\end{equation*}
where $\{\tilde e_{1}, \dots,\tilde e_{n-1}\}$ is a local orthonormal frame on $\Sigma_{r} = \{ f = r^{2}/4\}$.
\end{prop}
\begin{proof}
Note that there are Killing vector fields $\overline U_{a}$ for $g_{\alpha}$ on $S^{n-1}\times (r_{0},\infty)$ given by radially extending a basis for the Killing vector fields on the sphere. In particular, $\sL_{\overline U_{a}} g_{\alpha} = 0$ and $\Div (\sL_{\overline U_{a}} g_{\alpha}) - \frac 12 \nabla(\tr \sL_{\overline U_{a}}g_{\alpha})=0$. Furthermore it is not hard to see that by rescaling the $\overline U_{a}$ if necessary, we have that 
\begin{equation*}
\sum_{a=1}^{\frac{n(n-1)}{2}} \overline U_{a}\otimes\overline U_{a} = r^{2} \sum_{i=1}^{n-1} \overline e_{i}\otimes \overline e_{i},
\end{equation*}
where $\{\overline e_{1},\dots,\overline e_{n-1}\}$ is a local orthonormal frame for $S^{n-1}\times\{r\}$ with respect to $g_{\alpha}$. We may find vector fields $U_{a}$ on $M$ so that on the image of $F$, we have that $F_{*}\overline U_{a} =  U_{a}$ (we may extend them arbitrarily into the compact region, as only their asymptotic behavior will matter). Because $g$ is asymptotically conical, $|\sL_{U_{a}} g | = |\sL_{U_{a}} k| =  O(r^{-2\epsilon})$ and  
\begin{equation*}
\left|\Div (\sL_{U_{a}} g) - \frac 12 \nabla(\tr \sL_{U_{a}}g)\right| = \left|\Div (\sL_{U_{a}} k) - \frac 12 \nabla(\tr \sL_{U_{a}}k)\right|=O(r^{-2\epsilon}).
\end{equation*}
Because we have assumed that $F$ parametrizes the level sets of $f$, we have that
\begin{equation*}
\left[ \sqrt{f} \frac{X}{|X|^{2}}, U_{a}\right] = 0.
\end{equation*}
As such, 
\begin{equation*}
\begin{split}
[X,U_{a}] & = U_{a}\left( \frac{|X|^{2}}{\sqrt{f}} \right)\sqrt{f}\frac{X}{|X|^{2}}\\
& = U_{a}(|X|^{2}) \frac{X}{|X|^{2}}\\
& = - U_{a}(R) \frac{X}{|X|^{2}}.
\end{split}
\end{equation*}
Using \eqref{eq:bd-nabla-R}, we thus have that $|[X,U_{a}]| \leq |\nabla R| |U_{a}||X|^{-1} = O(r^{-2\epsilon})$. 
This may easily be used to show that $|\Delta U_{a} + D_{X} U_{a} - \frac 12 U_{a}| = O(r^{-2\epsilon})$. Finally, the tensorial identity follows readily from the asymptotics of the metric. 
\end{proof}
\begin{theo}
Suppose that $U$ is a vector field on $M$ with $|\sL_{U} g| = O(r^{-2\epsilon})$ and $|\Delta U + D_{X} U - \frac 12 U| = O(r^{-2\epsilon})$ for some $\epsilon < \frac {1}{\sqrt{2}}$. Then, there exists a vector field $W$ so that $\sL_{W} g = 0$, $[W,X] = 0$, $\bangle{W,X} = 0$ and $|W - U| \leq O(r^{-2\epsilon})$.
\end{theo}
\begin{proof}
Using Proposition \ref{prop:solve-for-vf-decay}, we may find a vector field $V$ so that 
\begin{equation*}
\Delta(U-V) + D_{X}(U-V) - \frac 12 (U-V) = 0
\end{equation*}
and $|V| = O(r^{-2\epsilon})$, $|DV| = O(r^{-2\epsilon})$. Setting $W = U-V$, we thus have that $|\sL_{W} g| = O(r^{-2\epsilon})$. Using Proposition \ref{prop:lich-eqn-barrier}, we thus see that $\sL_{W} g = 0$. This implies that $\Delta W + \Ric(W) = 0$, and combined with $\Delta W + D_{X} W - \frac 12 W = 0$, we thus see that $[W,X] = 0$. Finally, because $W$ is a Killing vector field
\begin{equation*}
\nabla(\sL_{W} f) = \sL_{W}( \nabla f) = [W,X] = 0. 
\end{equation*}
Thus $\sL_{W} f= \bangle{W,X}$ must be constant. However, $f$ attains its minimum somewhere in the compact region so in fact $\bangle{W,X} = 0$.
\end{proof}
Applying this to each of the approximate Killing vectors constructed above yields the following. 
\begin{coro}
There are vector fields $W_{a}$ for $a \in \left\{ 1 ,\dots,\frac{n(n-1)}{2}\right\}$ so that $\sL_{W_{a}} g = 0$, $[W_{a},X] = 0$ and $\bangle{W_{a},X} = 0$. Furthermore, $|W_{a}| = O(r)$ and 
\begin{equation*}
\sum_{a=1}^{\frac{n(n-1)}{2}}W_{a}\otimes W_{b} = r^{2}\sum_{i=1}^{n-1} \tilde e_{i}\otimes \tilde e_{i}+ O(r^{2-2\epsilon})
\end{equation*}
where $\{\tilde e_{1}, \dots,\tilde e_{n-1}\}$ is a local orthonormal frame on $\Sigma_{r} = \{ f = r^{2}/4\}$.
\end{coro}

This implies Theorem \ref{theo:main-unique-theo} as follows. The above corollary clearly implies that $(M,g)$ is rotationally symmetric, at least outside of some compact set. This is because we have shown that the Killing vectors $W_{a}$ span an $(n-1)$-dimensional space at each point in the asymptotic region. In particular, if $n=3$, this implies that the Cotton tensor vanishes outside of some compact set, while in dimensions $n\geq 4$, this implies that the Weyl tensor vanishes outside of some compact set. By the classical result of Bando, $(M,g)$ must be real analytic \cite{Bando:AnalyticRF}. Thus, we see that the Cotton tensor and Weyl tensor are also real analytic, and so if they vanish in an open set then they must vanish identically. This shows that $(M,g)$ must be locally conformally flat. However, this is well known to imply rotational symmetry, cf.\ \cite[Theorem 5.8 and 5.9]{CaoCatinoChenMantegazzaMazzieri:BachFlat} for a result that includes this statement as an obvious corollary. 

\appendix
\section{Expanding Bryant Solitons}

In this appendix, we describe the rotationally symmetric expanding solitons (with positive sectional curvature) constructed by Bryant in the unpublished note \cite{Bryant:Solitons}. In particular, we check below that they satisfy the conditions of Definition \ref{defi:asymp-cone-sol}. We remark that Bryant's family extends past the Gaussian (flat) soliton to continue into a family of negatively curved rotationally symmetric expanding gradient solitons, which we do not discuss here (see the discussion in \cite[Corollary 3]{Bryant:Solitons}).

It is standard (see, e.g., \cite[Section 2.3]{Petersen:RiemGeo}) that for a warped product metric of the form $g = dt^{2} + a(t)^{2} g_{S^{n-1}}$,
\begin{equation*}
\Ric = -(n-1)\frac{a''(t)}{a(t)}dt^{2} + ((n-2)-a(t)a''(t)-(n-2)a'(t)^{2})g_{S^{n-1}}.
\end{equation*}
In fact, we recall for later use that the metric $g$ has sectional curvature in the radial direction given by $-\frac{a''(t)}{a(t)}$ and for planes tangent to the orbits of rotation given by $\frac{1-(a'(t))^{2} } {a(t)^{2} }$. Furthermore, for a function $f(t)$, the Hessian of $f$ with respect to the metric $g$ is given by
\begin{equation*}
D^{2} f = f'' (t)dt^{2} + a(t)a' (t)f'(t) g_{S^{n-1}}.
\end{equation*}

Thus, we see that the soliton equations $2D^{2}f = g + 2\Ric$ are equivalent to the following family of ODEs:
\begin{equation*}
\begin{split}
2 f''(t) & = 1 - 2(n-1)\frac{a''(t)}{a(t)}\\
2 a(t) a'(t) f'(t) &  =  a(t)^{2}+ 2 ((n-2)-a(t)a''(t) - (n-2)a'(t)^{2}).
\end{split}
\end{equation*}
Supposing that there is a fixed point of the rotation, i.e., $t_{0} \in \RR$ so that $a(t_{0}) =0$ (by translating, we may assume that $t_{0} = 0$), the second soliton equation clearly implies that $a'(0) = \pm 1$. By reversing the $t$-variables if necessary, we thus may assume that $a'(t) > 0$ on $[0,T)$ for some $T>0$. In this region, it is convenient to change radial coordinates, from $t$ to $a = a(t)$. The metric in these coordinates may now be written 
\begin{equation*}
g = \frac{da^{2}}{\omega(a^{2})} + a^{2}g_{S^{n-1}}
\end{equation*}
where $\omega(a^{2})$ is defined implicitly by
\begin{equation*}
a'(t) = \sqrt{\omega(a(t)^{2})}.
\end{equation*}
The Ricci tensor in these coordinates is given by
\begin{equation*}
\Ric = -(n-1) \omega'(a^{2}) \frac{da^{2} }{ \omega(a^{2})} +((n-2) - \omega'(a^{2})a^{2} - (n-2) \omega(a^{2})) g_{S^{n-1}},
\end{equation*}
and the Hessian of $f(a^{2})$ by
\begin{equation*}\begin{split}
D^{2}f = & \left( 4 f''(a^{2}) a^{2}\omega(a^{2})+ 2 f'(a^{2})\omega(a^{2})  + 2 f'(a^{2})a^{2} \omega'(a^{2}) \right) \frac{da^{2}}{ \omega(a^{2})}  \\
& \qquad + 2f'(a^{2})  \omega(a^{2}) a^{2} g_{S^{n-1}}.
\end{split}\end{equation*}
In particular, the expanding soliton equations imply that the following system of ODEs must hold
\begin{equation}\label{eq:app-sol-eqns}
\begin{split}
 1 - 2(n-1) \omega'(s)  &  = 8 f''(s) s \omega(s)+ 4 f'(s)\omega(s)  + 4 f'(s)s \omega'(s)\\
 4 f'(s) s \omega(s) & = s + 2 ((n-2) - \omega'(s) s - (n-2) \omega(s)),
\end{split}
\end{equation}
where we have set $s=a^{2}$. Differentiating the second equation in $s$, we may eliminate the dependence on $f$ in the first equation, obtaining 
\begin{equation}\label{eq:app-eqn-h}
4 s^{2}\omega(s) \omega''(s) = 2(n-2)\omega(s)(\omega(s)-1)+s\omega'(s)(2s\omega'(s)-s-2(n-2)).
\end{equation}
\begin{lemm}[{\cite[Lemma 1]{Bryant:Solitons}}]
For $\omega(s)$ a positive solution of \eqref{eq:app-eqn-h}, defined for $s \in [0,M) \subset (0,\infty)$. Then, either $\omega\equiv 1$ or $\omega$ has at most one critical point in $(0,M)$ which is nondegenerate if it exists. Furthermore, if $\omega'(s_{0}) \geq 0$ and $\omega(s_{0}) >1$ for $s_{0}\in (0,M)$, then $\omega'(s)>$ for $s \in (s_{0},M)$. Similarly, if $\omega'(s_{0}) \leq 0$ and $\omega(s_{0}) <1$ then $\omega'(s) < 0$ on $(s_{0},M)$. 
\end{lemm}
To prove this, one may observe that \eqref{eq:app-eqn-h} shows that if $s_{0}$ is a critical point of $\omega$, then
\begin{equation*}
\omega''(s_{0}) = \frac{(n-2)(\omega(s_{0}) -1)}{2s_{0}^{2}}.
\end{equation*}
This shows that at a critical point of $h$, $\omega''(s_{0})>0$ is equivalent to $\omega(s_{0})>1$ and $\omega''(s_{0})<0$ is equivalent to $\omega(s_{0})<1$ ($\omega(s_{0})=1$ implies that $\omega\equiv 1$ by ODE uniqueness), so one may consider various cases to check the asserted properties.

\begin{prop}[{\cite[Proposition 4]{Bryant:Solitons}}]\label{prop:app-max-extend}
If $\omega(s)$ is a solution of \eqref{eq:app-eqn-h} with $\omega(0) = 1$ and $\omega'(0) < 0$ and $\omega(s)$ is defined on a maximally extended interval $[0,M) \subset [0,\infty)$, then necessarily $M = \infty$. 
\end{prop}

\begin{proof}
We first claim that if $M < \infty$, then $\lim_{s\nearrow M} \omega(s) = 0$. To see this, note that \eqref{eq:app-eqn-h} implies 
\begin{equation*}
\omega''(s) \geq -\frac{n-2}{2s^{2}} + \frac 12 (\omega'(s))^{2} - \frac 14 \omega'(s) \left(1+\frac{2(n-2)}{s}\right).
\end{equation*}
From this, it is clear that there is some $C>0$ so that if $\omega'(s) \leq -C$ for some $s \geq 1$, then $\omega''(s) > 0$. This implies that $\omega'(s)$ must be uniformly bounded from below on $[0,M)$. Because we have assumed that $M<\infty$, it must be that $\lim_{s\nearrow M} \omega(s) = 0$, otherwise we could extend the solution $\omega(s)$ past $M$. 

Now, \eqref{eq:app-eqn-h} also implies that for $s >0$, then 
\begin{equation*}
\omega''(s) > - \frac{n-2}{2s^{2}} - \frac 14 \frac{\omega'(s)}{\omega(s)}.
\end{equation*}
Integrating from $s_{0}$ to $s < M$, this implies that 
\begin{equation*}
\omega'(s) -\omega'(s_{0}) > -\frac{n-2}{2}\left( \frac{1}{s_{0}} - \frac{1}{s}\right) + \frac 1 4 \log\left( \frac{\omega(s_{0})}{\omega(s)}\right).
\end{equation*}
Letting $s\nearrow M$, the left-hand side must tend to infinity, because $\lim_{s\nearrow M}\omega(s) = 0$, but the right-hand side is bounded above, a contradiction. 
\end{proof}

\begin{lemm}\label{lemm:app-o1-decay}
For $\omega(s)$ a solution of \eqref{eq:app-eqn-h} with $\omega(0) =1$ and $\omega'(0) < 0$, we have that $\omega'(s),\omega''(s) = o(1)$.
\end{lemm}
\begin{proof}
Rewriting \eqref{eq:app-eqn-h} as 
\begin{equation*}
\omega''(s) = \frac{(n-2)(\omega(s)-1)}{2s^{2}} - \frac{\omega'(s)(s+2(n-2))}{4s \omega(s)} + \frac{(\omega'(s))^{2}}{2 \omega(s)},
\end{equation*}
we see that for a fixed $\delta>0$, there is $s_{0}=s_{0}(\delta)$ large enough so that if $\omega'(s) < -\delta$ for $s \geq s_{0}$, then $\omega''(s) > 0$. On the other hand, we must be able to find $s_{1} > s_{0}$ so that $\omega'(s_{1}) > - \delta$ (otherwise $\omega(s)$ could not converge). As such, for $s \geq s_{1}$, $\omega'(s) \geq -\delta$ (we have just shown that $-\delta$ is a barrier for $\omega'(s)$). This clearly shows that $\omega'(s) = o(1)$. Using this in \eqref{eq:app-eqn-h} gives $\omega''(s) = o(1)$. 
\end{proof}

\begin{coro}[{\cite[Corollary 2]{Bryant:Solitons}}]\label{coro:app-all-time-monotone}
A solution of \eqref{eq:app-eqn-h} with $\omega(0) =1$ and $\omega'(0) < 0$ exists for all $s \geq 0$ and is monotonically decreasing with a positive lower bound. 
\end{coro}

\begin{proof}
As in the proof of Proposition \ref{prop:app-max-extend}, 
\begin{equation*}
\omega'(s) -\omega'(s_{0}) > -\frac{n-2}{2}\left( \frac{1}{s_{0}} - \frac{1}{s}\right) + \frac 1 4 \log\left( \frac{\omega(s_{0})}{\omega(s)}\right).
\end{equation*}
By the previous lemma, we have that $\omega'(s) - \omega'(s_{0}) \leq -\omega'(s_{0}) = o(1)$. Thus, it cannot happen that $\omega(s) \to 0$ as $s\to\infty$. 
\end{proof}

Now, we show that $\omega(s)$ and $f(s)$ agree with their formal asymptotic expansions up to second order; this will allow us to study the rate at which the Bryant solitons approach a cone.
\begin{prop}
For a solution of  \eqref{eq:app-eqn-h} with $\omega(0) =1$ and $\omega'(0) < 0$, by Corollary \ref{coro:app-all-time-monotone}, there is some $\alpha \in [0,1)$ so that $\lim_{s\to\infty } \omega(s) = 1-\alpha$. With this choice of $\alpha$, we have the asymptotic expansion of $\omega(s)$,
\begin{equation*}
\omega(s) = 1-\alpha + \frac{2(n-2)\alpha(1-\alpha)}{s} + \varphi(s),
\end{equation*}
where $\varphi(s)$ satisfies $\varphi(s) = O(s^{-2})$, $\varphi'(s) ,\varphi''(s) = O(s^{-3})$. Furthermore, we have that $f(s)$ satisfies the expansion (up to addition of a constant) 
\begin{equation*}
f(s) = \frac{s}{4(1-\alpha)} + \psi(s),
\end{equation*}
where $\psi(s) = O(s^{-1})$, $\psi'(s) = O(s^{-2})$, and $\psi''(s) = O(s^{-3})$. 
\end{prop}
\begin{proof}
By \eqref{eq:app-eqn-h}, we have that
\begin{equation*}
 \omega''(s) + \frac 1 4 \omega'(s) > -  \frac{C}{s^{2}}.
\end{equation*}
We may use an integrating factor to rewrite this as
\begin{equation*}
\frac{d}{ds}\left( e^{s/4} \omega'(s) \right) \geq - \frac{C}{s^{2}} e^{s/4}.
\end{equation*}
Integrating from $1$ to $s$ thus yields 
\begin{equation*}
e^{s/4}\omega'(s) - \omega'(1) \geq - C \int_{1}^{s} \frac{e^{x/4}}{x^{2}} dx.
\end{equation*}
Now, because
\begin{equation*}\begin{split}
\int_{1}^{s} \frac{e^{(x-s)/4}}{x^{2}}dx & = \int_{1}^{s/2} \frac{e^{(x-s)/4}}{x^{2}}dx + \int_{s/2}^{s} \frac{e^{(x-s)/4}}{\tau^{2}}dx\\
 & \leq \int_{1}^{s/2} e^{(x-s)/4} dx +  \frac{4}{s^{2}} \int_{s/2}^{s} e^{(x-s)/4} dx\\
 & \leq 4 \left(e^{-s/8}-e^{(1-s)/4}\right) + \frac {16}{s^{2}} \left(1 - e^{-s/8}\right) = O(s^{-2}),
\end{split}\end{equation*}
we have that $\omega'(s) = O(s^{-2})$. This implies $\omega(s)-1+\alpha = O(s^{-1})$ and from \eqref{eq:app-eqn-h} it is not hard to see that also $\omega''(s) = O(s^{-2})$. 

We now define a function $\varphi(s)$ by 
\begin{equation*}
\omega(s) = 1-\alpha + \frac{2(n-2)\alpha(1-\alpha)}{s}+ \varphi(s).
\end{equation*}
Here, the choice of second order term comes from formally expanding $\omega(s)$ in a power series in $s^{-k}$ and solving for the $s^{-1}$ term (the power series does not converge, cf.\ \cite[Remark 11]{Bryant:Solitons}, so we are simply using the truncated expansion to cancel the highest order term in the ODE). By the above asymptotics of $\omega(s)$, we see that $\varphi(s) = O(s^{-1})$ and $\varphi'(s),\varphi''(s) = O(s^{-2})$. Using this, one may show (as above, except the ODE for $\varphi(s)$ decays one order faster in $s$, as we have just explained)
\begin{equation*}
\varphi''(s) + \frac 1 4 \varphi'(s) \geq - \frac{C}{s^{3}},
\end{equation*}
and then the same argument implies that $\varphi(s) = O(s^{-2})$ and $\varphi'(s),\varphi''(s) = O(s^{-3})$, as desired.

Now, by the bottom line of \eqref{eq:app-sol-eqns}, we see that 
\begin{equation*}\begin{split}
4 \left[ f(s) - \frac{s}{4(\alpha-1)}\right]' \omega(s) &  = \frac{1-\alpha-\omega(s)}{1-\alpha} + 2 \left( \frac{n-2}{s}(1-\omega(s)) - \omega' (s)\right)\\
& = - \frac{\varphi(s)}{1-\alpha} - \frac{4(n-2)^{2}\alpha(1-\alpha)}{s^{2}} \\
& \qquad - \frac{2(n-2)\varphi(s)}{s} + \frac{2(n-2)\alpha(1-\alpha)}{s^{2}} - \varphi'(s)\\
& =: 4  \psi'(s) \omega(s).
\end{split}
\end{equation*}
Here, we may choose $\psi(s)$ so that $\psi(s)\to 0$ as $s \to \infty$. In particular, we easily see that $\psi(s) = O(s^{-1})$, $\psi'(s) = O(s^{-2})$, $\psi''(s) = O(s^{-3})$, as desired. 
\end{proof}
It is clear from the proof that it is possible to show that $\omega(s)$ and $f(s)$ agree with their formal power series at infinity up to any finite number of terms. However, as remarked above, the power series does not converge.

\begin{prop}\label{prop:app-bry-are-asymp-cone}
Each of the solutions $\omega(s)$ of \eqref{eq:app-eqn-h} with $\omega(0)=0$ and $\omega'(0) < 0$ define a rotationally symmetric soliton with positive sectional curvature that is asymptotically conical as a soliton, in the sense of Definition \ref{defi:asymp-cone-sol}.

\end{prop}
\begin{proof}
We have just shown that any solution of \eqref{eq:app-eqn-h} with $\omega(0)=0$ and $\omega'(0) < 0$ exists for all time and is monotonically decreasing with a positive lower bound. Fix $\alpha \in [0,1)$ so that $\lim_{s\to\infty}\omega(s) = 1-\alpha$. As the radial sectional curvature is $-\omega'(a^{2})$ and the sectional curvatures tangent to the orbits of rotations are $\frac{1-\omega(a^{2})}{a^{2}}$, these solutions have positive sectional curvature. That the soliton is asymptotically conical as a soliton follows readily from the asymptotics of $f(s)$ and $\omega(s)$ in the previous proposition.
%
%
%
\end{proof}

\bibliography{biblio} 
\bibliographystyle{amsalpha}
\end{document}